\documentclass[letterpaper, reqno,11pt]{article}
\usepackage[margin=1.0in]{geometry}
\usepackage{color,latexsym,amsmath,amssymb, amsthm, graphicx,subfigure,revsymb, enumerate}
\usepackage[percent]{overpic}

\numberwithin{equation}{section}

\newcommand{\RR}{\mathbb{R}}
\newcommand{\CC}{\mathbb{C}}

\newcommand{\eps}{\varepsilon}

\newcommand{\baseField}{F}
\newcommand{\extendedField}{E}
\newcommand{\KP}{F\mathbf{P}}
\newcommand{\LP}{E\mathbf{P}}
\newcommand{\klein}{K}
\newcommand{\lie}{L}

\newtheorem{thm}{Theorem}[section]
\newtheorem{lem}[thm]{Lemma}
\newtheorem{prop}[thm]{Proposition}
\newtheorem{cor}[thm]{Corollary}

\newtheorem*{linearPencilLem}{Lemma \ref{linearPencilProp}}

\theoremstyle{remark}
\newtheorem{defn}[thm]{Definition}
\newtheorem{example}[thm]{Example}

\newtheorem{rem}[thm]{Remark}

\begin{document}
\pagenumbering{arabic}
\title{Sphere tangencies, line incidences, and Lie's line-sphere correspondence}
\author{Joshua Zahl\thanks{University of British Columbia, jzahl@math.ubc.ca}}
\maketitle

\begin{abstract}
Two spheres with centers $p$ and $q$ and signed radii $r$ and $s$ are said to be in contact if $|p-q|^2 = (r-s)^2$. Using Lie's line-sphere correspondence, we show that if $\baseField$ is a field in which $-1$ is not a square, then there is an isomorphism between the set of spheres in $\baseField^3$ and the set of lines in a suitably constructed Heisenberg group that is embedded in $(\baseField[i])^3$; under this isomorphism, contact between spheres translates to incidences between lines.  

In the past decade there has been significant progress in understanding the incidence geometry of lines in three space. The contact-incidence isomorphism allows us to translate statements about the incidence geometry of lines into statements about the contact geometry of spheres. This leads to new bounds for Erd\H{o}s' repeated distances problem in $\baseField^3$, and improved bounds for the number of point-sphere incidences in three dimensions. These new bounds are sharp for certain ranges of parameters.
\end{abstract}

\section{Introduction}\label{introSection}

Let $\baseField$ be a field in which $-1$ is not a square. For each quadruple $(x_1,y_1,z_1,r_1)\in \baseField^4$, we associate the (oriented) sphere $S_1\subset \baseField^3$ described by the equation $(x-x_1)^2 + (y-y_1)^2 + (z-z_1)^2 = r_1^2$. We say two oriented spheres $S_1,S_2$ are in ``contact'' if
\begin{equation}\label{defnOfContact}
(x_1-x_2)^2 + (y_1-y_2)^2 + (z_1-z_2)^2 = (r_1-r_2)^2.
\end{equation}  
If $\baseField=\RR$, then this has the following geometric interpretation: 
\begin{itemize}
\item If $r_1$ and $r_2$ are non-zero and have the same sign, then \eqref{defnOfContact} describes internal tangency. 
\item If $r_1$ and $r_2$ are non-zero and have opposite signs, then \eqref{defnOfContact} describes external tangency.
\end{itemize}

In this paper we will explore the following type of extremal problem in combinatorial geometry: Let $n$ be a large integer. If $\mathcal{S}$ is a set of $n$ oriented spheres in $\baseField^3$ (possibly with some additional restrictions), how many pairs of spheres can be in contact? This question will be answered precisely in Theorem \ref{spheresInGeneralFieldThm} below.

When $\baseField=\RR$, variants of this problem have been studied extensively in the literature \cite{ApfelbaumSharir, CEGSW, KMSS, Z}. For example, Erd\H{o}s' repeated distances conjecture in $\RR^3$ \cite{E60} asserts that $n$ spheres in $\RR^3$ of the same radius must determine $O(n^{4/3})$ tangencies. The current best-known bound is $O(n^{295/197+\eps})$ in $\RR^3$ \cite{Z2}. In Theorem \ref{unitDistancesR3} we will establish the weaker bound $O(n^{3/2})$ which is valid in all fields for which $-1$ is not a square.

Before discussing this problem further, we will introduce some additional terminology. Let $i$ be a solution to $x^2+1=0$ in $\baseField$ and let $\extendedField=\baseField[i]$. Each element $\omega\in \extendedField$ can be written uniquely as $\omega = a + bi,$ with $a,b\in \baseField$. We define the involution $\bar\omega = a-bi$, and we define 
\begin{equation*}
\operatorname{Re}(\omega) = \frac{1}{2}(\omega+\bar\omega),\quad \operatorname{Im}(\omega) = \frac{1}{2}(\omega-\bar\omega).
\end{equation*} 
We define the Heisenberg group
\begin{equation}\label{defnHeisenberg}
\mathbb{H} = \{(x,y,z)\in \extendedField^3\colon \operatorname{Im}(z) = \operatorname{Im}(x\bar y)\}.
\end{equation}
The Heisenberg group contains a four-parameter family of lines. In particular, if $a,b,c,d\in \baseField$, then $\mathbb{H}$ contains the line
\begin{equation}\label{eqnOfLineInH}
\{ (0,c+di,a)+(1,b,c-di)t:\ t\in \extendedField\},
\end{equation}
and every line contained in $\mathbb{H}$ that is not parallel to the $xy$ plane is of this form. If $a_1,b_1,c_1,d_1$ and $a_2,b_2,c_2,d_2$ are elements of $\baseField$, then the corresponding lines are coplanar (i.e. they either intersect or are collinear) precisely when
\begin{equation}\label{linesCollinear}
(c_1 + d_1 i -c_2 - d_2i,\ a_1-a_2)\wedge (b_1-b_2,\ c_1-d_1i - c_2 + d_2i) = 0.
\end{equation}
The Heisenberg group $\mathbb{H}$ has played an important role in studying the Kakeya problem \cite{KLT, KZ, MT}. More recently, it has emerged as an important object in incidence geometry \cite{GuthBook}. See \cite{TaoBlog} (and in particular, the discussion surrounding Proposition 5) for a nice introduction to the Heisenberg group in this context.

Our study of contact problems for spheres in $\baseField^3$ begins by observing that the contact geometry of (oriented) spheres in $\baseField^3$ is isomorphic to the incidence geometry of lines in $\mathbb{H}$ that are not parallel to the $xy$ plane. Concretely, to each oriented sphere $S_1\subset \baseField^3$ centered at $(x_1,y_1,z_1)$ with radius $r_1$, we can associate a line $\ell$ of the form \eqref{eqnOfLineInH}, with $a=-z_1-r_1$, $b = -z_1+r_1$, $c = -x_1$, $d = -y_1$. Two oriented spheres $S_1$ and $S_2$ are in contact if and only if the corresponding lines $\ell_1$ and $\ell_2$ are coplanar. We will discuss this isomorphism and its implications in Section \ref{lieLineSphereSec}. This isomorphism is not new---it is known classically as Lie's line-sphere correspondence (it is also similar in spirit to previous reductions that relate questions in incidence geometry to problems about incidences between lines in three space \cite{Rudnev, deZ, RS} ). However, we are not aware of this isomorphism previously being used in the context of combinatorial geometry.

The isomorphism is interesting for the following reason. In the past decade there has been significant progress in understanding the incidence geometry of lines in $\extendedField^3$. This line of inquiry began with Dvir's proof of the finite field Kakeya problem \cite{dvir} and Guth and Katz's proof of the joints conjecture \cite{GKJoints}, as well as subsequent simplifications and generalizations of their proof by Quilodr\'an and independently by Kaplan, Sharir, and Shustin \cite{KSS}. More recently, and of direct relevance to the problems at hand, Guth and Katz resolved the Erd\H{o}s distinct distances problem in $\RR^2$ (up to the endpoint) by developing new techniques for understanding the incidence geometry of lines in $\RR^3$. Some of these techniques were extended to all fields by Koll\'ar \cite{Kollar} and by Guth and the author \cite{GZ}. The isomorphism described above allows us to translate these results about the incidence geometry of lines into statements about the contact geometry of spheres. We will describe a number of concrete statements below.

To begin exploring the contact geometry of oriented spheres, we should ask: what arrangements of oriented spheres in $\baseField^3$ have many pairs of spheres that are in contact? The next example shows that there are sets of spheres in $\baseField^3$ so that \emph{every} pair of oriented spheres is in contact.

\begin{example}\label{pencilExample}
Let $(x,y,z,r)\in \baseField^4$ and let $(u,v,w)\in \baseField^3$ with $u^2 + v^2 + w^2 = s^2$ for some nonzero $s\in \baseField$. Consider the set of oriented spheres $\mathfrak{P}=\{ (x + ut, y+vt, z+wt, r+st)\colon t\in \baseField\}$. Every pair of spheres from this set is in contact. See Figure \ref{pencilPic}.
\end{example}

\begin{figure}
\centering
\begin{minipage}{.5\textwidth}
  \centering
  \includegraphics[width=0.9\textwidth]{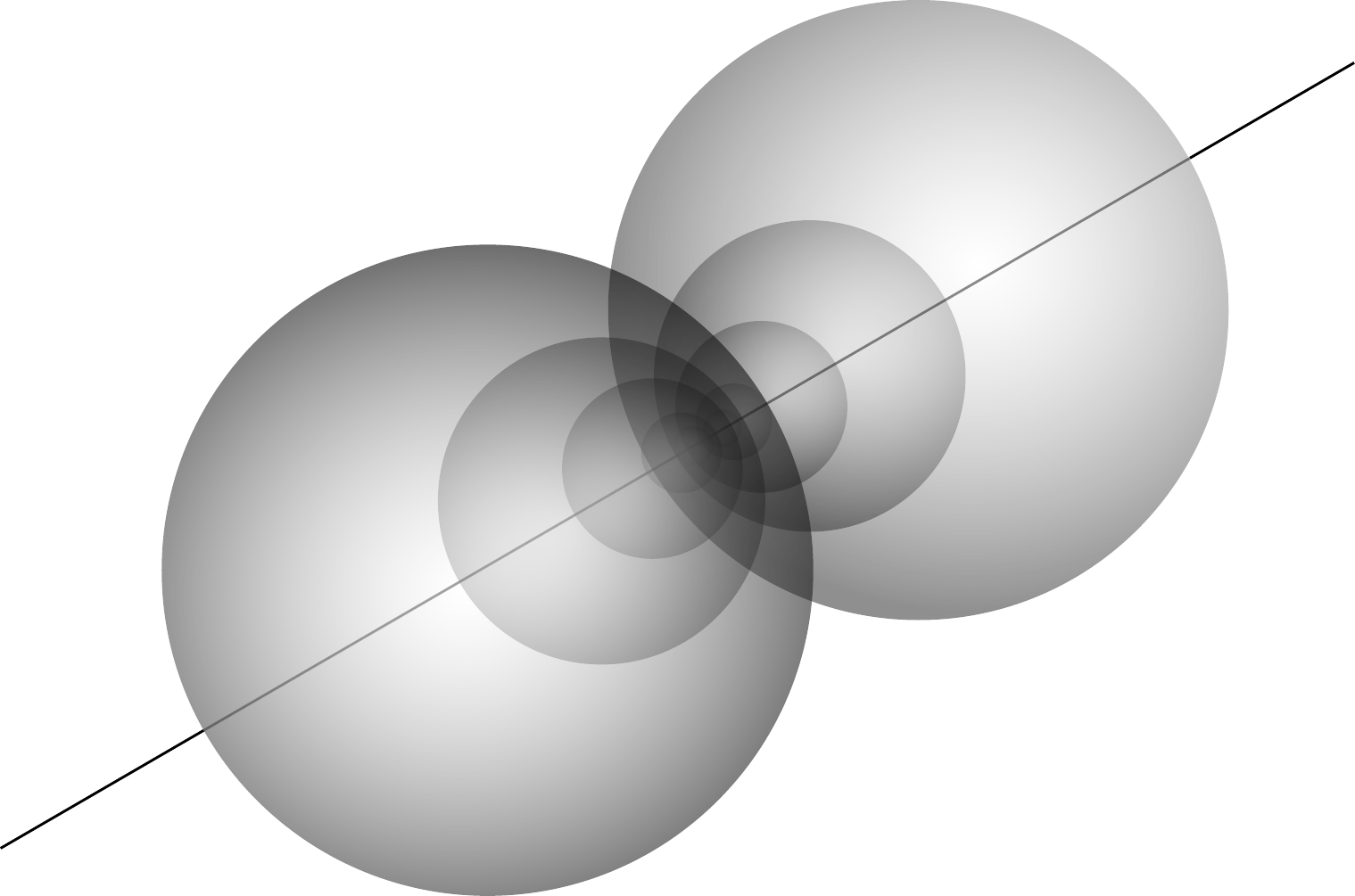}
  \label{fig:test1}
\end{minipage}%
\begin{minipage}{.5\textwidth}
  \centering
  \includegraphics[width=0.9\textwidth]{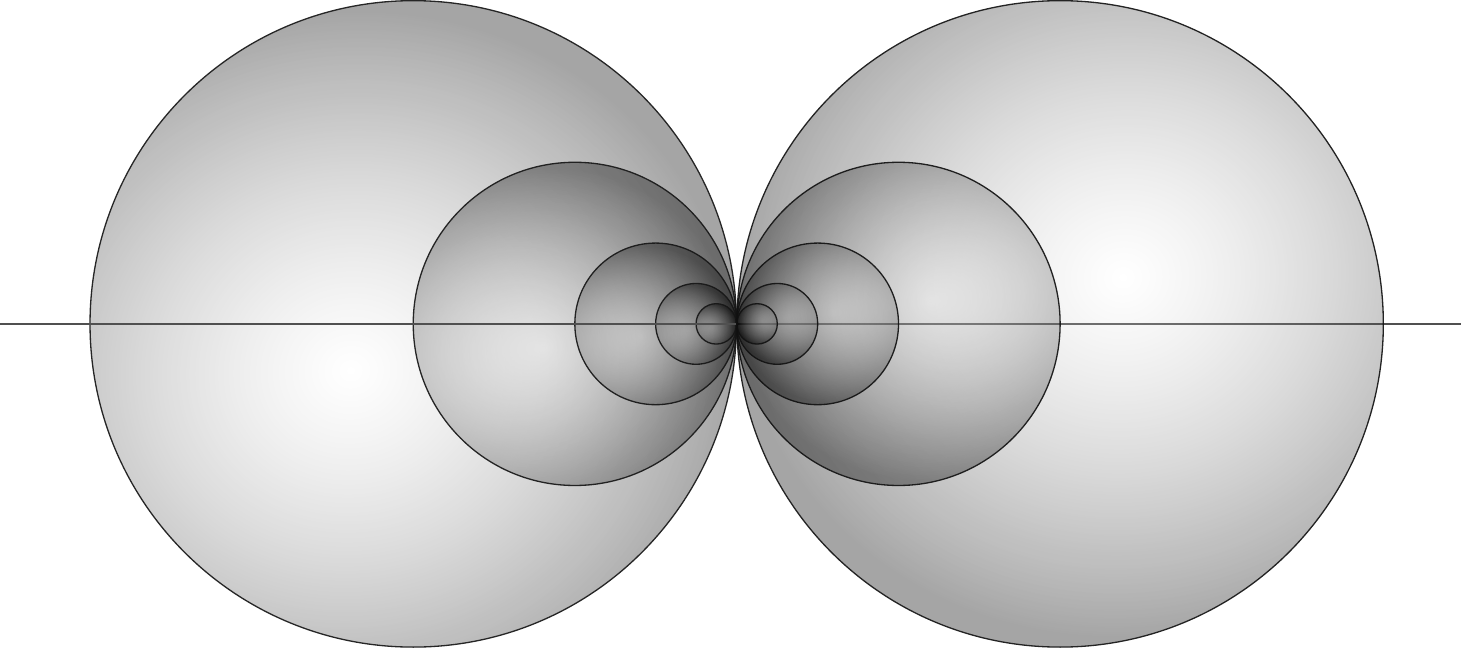}
  \label{fig:test2}
\end{minipage}
\caption{A pencil of contacting spheres. The left and right images are two different perspectives of the same set. Note that there is precisely one sphere of radius 0. If $\baseField=\RR$, then each sphere on one side of the center point has positive radius, while each sphere on the other side has negative radius. }\label{pencilPic}
\end{figure}

\begin{defn}\label{pencilDefn}
Let $\mathfrak{P}$ be a set of spheres, every pair of which is in contact. If $\mathfrak{P}$ is maximal (in the sense that no additional spheres can be added to $\mathfrak{P}$ while maintaining this property), then  $\mathfrak{P}$ is called a ``pencil of contacting spheres.'' If $\mathcal{S}$ is a set of oriented spheres and $\mathfrak{P}$ is a pencil of contacting spheres, we say that $\mathfrak{P}$ is $k$-rich (with respect to $\mathcal{S}$) if at least $k$ spheres from $\mathcal{S}$ are contained in $\mathfrak{P}$. We say $\mathfrak{P}$ is \emph{exactly} $k$-rich if exactly $k$ spheres from $\mathcal{S}$ are contained in $\mathfrak{P}$.
\end{defn}

\begin{rem}
If we identify each sphere in a pencil of contacting spheres with its coordinates $(x,y,z,r)$, then the corresponding points form a line in $\baseField^4$. If we identify each sphere in a pencil of contacting spheres with its corresponding line in $\mathbb{H}$, then the resulting family of lines are all coplanar and pass through a common point (possibly at infinity\footnote{If two lines are parallel, then we say these lines pass through a common point at infinity.}). This will be discussed further in Section \ref{pencilSection}.
\end{rem}

The next result says that any two elements from a pencil of contacting spheres uniquely determine that pencil.
\begin{lem}\label{linearPencilProp}
Let $S_1$ and $S_2$ be distinct oriented spheres that are in contact. Then the set of oriented spheres that are in contact with $S_1$ and $S_2$ is a pencil of contacting spheres. 
\end{lem}

We will defer the proof of Lemma \ref{linearPencilProp} to Section \ref{pencilSection}. Example \ref{pencilExample} suggests that rather than asking how many spheres from $\mathcal{S}$ are in contact, we should instead ask how many $k$-rich pencils can be determined by $\mathcal{S}$---each pencil that is exactly $k$ rich determines $\binom{k}{2}$ pairs of contacting spheres. We begin with the case $k=2$. Since each pair of spheres can determine at most one pencil, a set $\mathcal{S}$ of oriented spheres determines at most $\binom{|\mathcal{S}|}{2}$ 2-rich pencils of contacting spheres. The next example shows that in general we cannot substantially improve this estimate, because there exist configurations of $n$ oriented spheres that determine $\frac{n^2}{4}$  $2$-rich pencils. 

\begin{example}\label{conicExample}
Let $S_1,S_2$ and $S_3$ be three spheres in $\baseField^3$, no two of which are in contact. Suppose that at least three spheres are in contact with each of $S_1,S_2,$ and $S_3$, and denote this set of spheres by $\mathfrak{C}$. Let $\mathfrak{C}^\prime$ be the set of all spheres that are in contact with every sphere from $\mathfrak{C}$. Then every sphere from $\mathfrak{C}$ is in contact with every sphere from $\mathfrak{C}^\prime$ and vice-versa. Furthermore, Lemma \ref{linearPencilProp} implies that no two spheres from $\mathfrak{C}$ are in contact, and no two spheres from $\mathfrak{C}^\prime$ are in contact, so the spheres in $\mathfrak{C}\cup \mathfrak{C}^\prime$ do not determine any 3-rich pencils of contacting spheres. This means that if $\mathcal{S}\subset \mathfrak{C}$ and $\mathcal{S}^\prime\subset \mathfrak{C}^\prime$, then $\mathcal{S}\cup\mathcal{S}^\prime$ determines $|\mathcal{S}|\ |\mathcal{S}^\prime|$ 2-rich pencils of contacting spheres. See Figure \ref{complimentaryConicSectionsPic}.
\end{example}

\begin{defn}\label{conicDefn}
Let $\mathfrak{C}$ and $\mathfrak{C}^\prime$ be two sets of oriented spheres, each of cardinality at least three, with the property that each sphere from $\mathfrak{C}$ is in contact with each sphere from $\mathfrak{C}^\prime$, and no two spheres from the same set are in contact. If $\mathfrak{C}$ and $\mathfrak{C}^\prime$ are maximal (in the sense that no additional spheres can be added to $\mathfrak{C}$ or $\mathfrak{C}^\prime$ while maintaining this property), then $\mathfrak{C}$ and $\mathfrak{C}^\prime$ are called a ``pair of complimentary conic sections.''
\end{defn}

\begin{figure}[h!]
 \centering
\begin{overpic}[width=\textwidth]{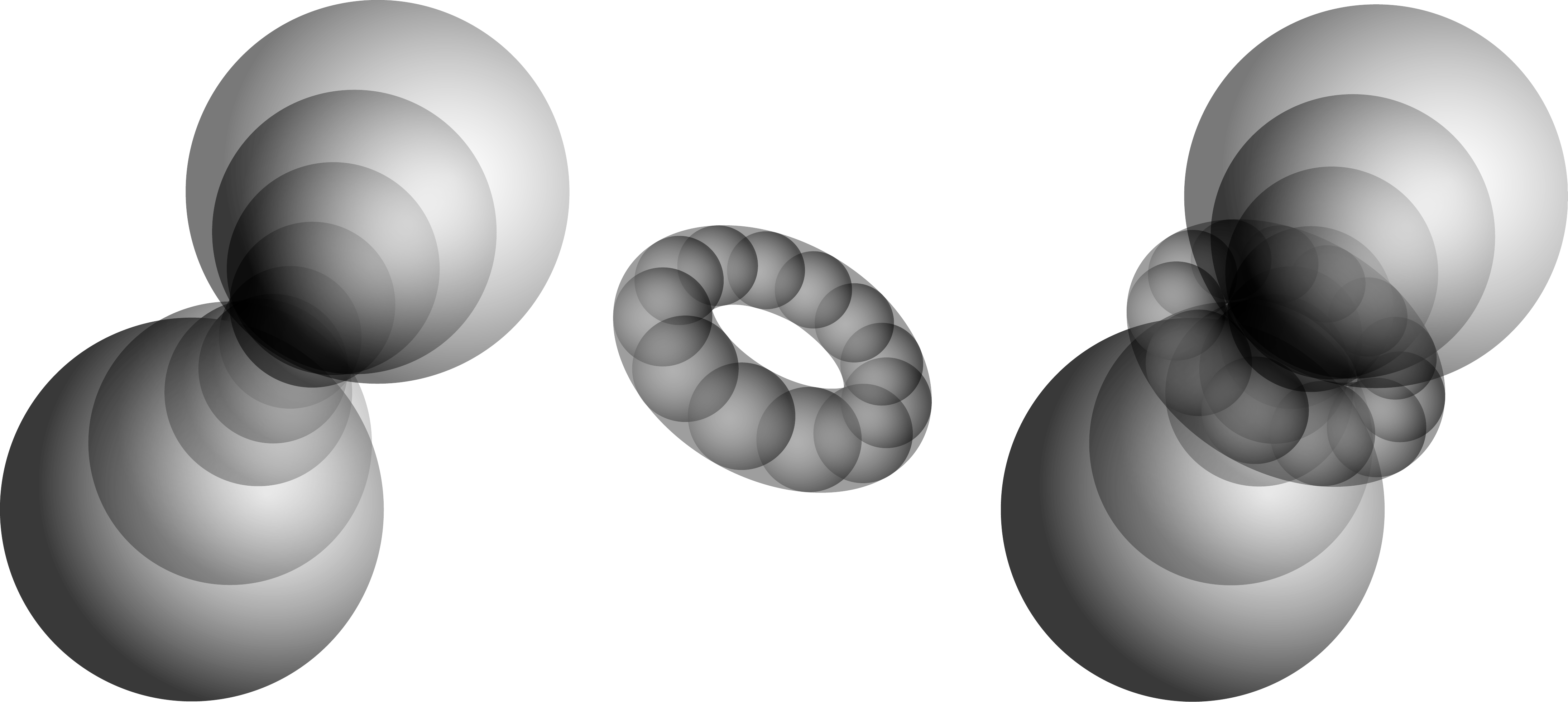}
\end{overpic}
 \caption{A pair of complimentary conic sections. The sets $\mathfrak{C}$ and $\mathfrak{C}^\prime$ are shown individually on the left and center, respectively, and $\mathfrak{C}\cup \mathfrak{C}^\prime$ is shown on the right. If $\baseField=\RR$, then (in this example) all of the spheres in $\mathfrak{C}$ have positive radius, and all of the spheres in $\mathfrak{C}^\prime$ have negative radius. In particular, while certain pairs of spheres in $\mathfrak{C}$ (resp. $\mathfrak{C}^\prime$) are (externally) tangent, no pairs of spheres are in contact in the sense of \eqref{defnOfContact}. Note that while the centers of the spheres in $\mathfrak{C}$ are contained in a line, the corresponding quadruples $(x_1,y_1,z_1,r_1)$ are contained in an irreducible degree-two curve.}\label{complimentaryConicSectionsPic}
\end{figure}

\begin{rem}
If we identify each sphere in a pair of complimentary conic sections with its coordinates $(x,y,z,r)$, then the corresponding points are precisely the $\baseField$-points on a pair of conic sections in $\baseField^4$, neither of which are lines. If we identify each sphere in a pair of complimentary conic sections with its corresponding line in $\mathbb{H}$, then the resulting families of lines are contained in two rulings of a doubly-ruled surface in $\extendedField^3$. Note, however, that the two rulings of this doubly-ruled surface contain additional lines that do not come from the pair of complimentary conic sections, since not all lines in the rulings will be contained in $\mathcal{H}$. This will be discussed further in Section \ref{conicsSection}.
\end{rem}

We are now ready to state our main result. Informally, it asserts that the configurations described in Examples \ref{pencilExample} and \ref{conicExample} are the only way that many pairs of spheres in $\baseField^3$ can be in contact.

\begin{thm}\label{spheresInGeneralFieldThm}
Let $\baseField$ be a field in which $-1$ is not a square. Let $\mathcal{S}$ be a set of $n$ oriented spheres in $\baseField^3$, with $n\leq (\operatorname{char} \baseField)^2$ (if $\baseField$ has characteristic zero then we impose no constraints on $n$). Then for each $3\leq k\leq n$, $\mathcal{S}$ determines $O(n^{3/2}k^{-3/2})$ $k$-rich pencils of contacting spheres. Furthermore, at least one of the following two things must occur:
\begin{itemize}
\item There is a pair of complimentary conic sections $\mathfrak{C},\mathfrak{C}^\prime$ so that 
$$
|\mathfrak{C}\cap \mathcal{S}|\geq \sqrt n\quad\textrm{and}\quad |\mathfrak{C}^\prime\cap \mathcal{S}|\geq \sqrt n.
$$
\item $\mathcal{S}$ determines $O(n^{3/2})$ $2$-rich pencils of contacting spheres.
\end{itemize}
\end{thm}
\noindent Theorem \ref{spheresInGeneralFieldThm} leads to new bounds for Erd\H{o}s' repeated distances problem in $\baseField^3$.

\begin{thm}\label{unitDistancesR3}
Let $\baseField$ be a field in which $-1$ is not a square. Let $r\in \baseField\backslash \{0\}$ and let $\mathcal{P}\subset \baseField^3$ be a set of $n$ points in $\baseField^3$, with $n\leq (\operatorname{char} \baseField)^2$ (if $\baseField$ has characteristic zero then we impose no constraints on $n$). Then there are $O(n^{3/2})$ pairs $(x_1,y_1,z_1),(x_2,y_2,z_2)\in \mathcal{P}$ satisfying 
\begin{equation}\label{distanceR}
(x_1-x_2)^2 + (y_1-y_2)^2 + (z_1-z_2)^2 = r^2.
\end{equation} 
\end{thm}

As discussed above, when $F=\RR$ the conjectured bound is $O(n^{4/3})$ and the current best known bound is $O(n^{\frac{3}{2}-\frac{1}{394}+\eps})$ \cite{Z2}. For general fields in which $-1$ is not a square, the previous best-known bound was $O(n^{5/3})$. 

\begin{rem}\label{noImprovement}
The bound $O(n^{3/2})$ given above cannot be improved without further restrictions on $n$. Indeed, if we select $\mathcal{P}$ to be a set of $p^2$ points in $\mathbb{F}_p^3$, then by pigeonholing there exists an element $r\in \mathbb{F}_p\backslash\{0\}$ so that there are\footnote{Some care has to be taken when selecting the points in $\mathcal{P}$ to ensure that not too many pairs of points satisfy $(x_1-x_2)^2 + (y_1-y_2)^2 + (z_1-z_2)^2 = 0$, but this is easy to achieve.} roughly $ p^3$ solutions to \eqref{distanceR}. When $n$ is much smaller than $p^2$ (e.g. if $n\leq p$) it seems plausible that there should be $O(n^{4/3})$ solutions to \eqref{distanceR}.
\end{rem}

\begin{rem}
The requirement that $-1$ not be a square is essential, since if $-1$ is a square in $\baseField$ then for each $r\in \baseField\backslash \{0\}$, the sphere $\{(x,y,z)\in F^3\colon x^2+y^2+z^2 = r^2\}$ is doubly-ruled by lines (see e.g. \cite[Lemma 6]{Rudnev2}). It is thus possible to find an arrangement of $n/2$ spheres of radius $r$, all of which contain a common line $\ell$. Let $\mathcal{P}=\mathcal{P}_1\cup\mathcal{P}_2$, where $\mathcal{P}_1$ is the set of centers of the $n/2$ spheres described above and $\mathcal{P}_2$ is a set of $n/2$ points on $\ell$. Then $\mathcal{P}$ has $n^2/4$ pairs of points that satisfy \eqref{distanceR}.
\end{rem}

Theorem \ref{unitDistancesR3} can also be used to prove new results for the incidence geometry of points and spheres in $\baseField^3$. In general, it is possible for $n$ points and $n$ spheres in $\baseField^3$ to determine $n^2$ points-sphere incidences. For example, we can place $n$ points on a circle in $\baseField^3$ and select $n$ spheres which contain that circle. The following definition, which is originally due to Elekes and T\'oth \cite{ET} in the context of hyperplanes, will help us quantify the extent to which this type of situation occurs.

\begin{defn}
Let $\baseField$ be a field, let $\mathcal{P}\subset \baseField^3$, and let $\eta>0$ be a real number. A sphere $S\subset \baseField^3$ is said to be $\eta$-non-degenerate (with respect to $\mathcal{P}$) if for each plane $H\subset \baseField^3$ we have $|\mathcal{P}\cap S\cap H|\leq \eta|\mathcal{P}\cap S|$. 
\end{defn}
\begin{thm}[Point-sphere incidences]\label{pointSphereThm}
Let $\baseField$ be a field in which $-1$ is not a square. Let $\mathcal{S}$ be a set of $n$ spheres (of nonzero radius) in $\baseField^3$ and let $\mathcal{P}$ be a set of $n$ points in $\baseField^3$, with $n\leq (\operatorname{char} \baseField)^2$ (if $\baseField$ has characteristic zero then we impose no constraints on $n$). Let $\eta>0$ and suppose the spheres in $\mathcal{S}$ are $\eta$-non-degenerate with respect to $\mathcal{P}$. Then there are $O(n^{3/2})$ incidences between the points in $\mathcal{P}$ and the spheres in $\mathcal{S}$, where the implicit constant depends only on $\eta$. 
\end{thm}
In \cite{ApfelbaumSharir},  Apfelbaum and Sharir proved that $m$ points and $n$ $\eta$-non-degenerate spheres in $\RR^3$ determine $O^*(m^{8/11}n^{9/11}+mn^{1/2})$ incidences, where the notation $O^*(\cdot)$ suppresses sub-polynomial factors. When $m=n$, this bound simplifies to $O^*(n^{17/11})$. Thus in the special case $m=n$, Theorem \ref{pointSphereThm} both strengthens the incidence bound of  Apfelbaum and Sharir and extends the result from $\RR$ to fields in which $-1$ is not a square. A construction analogous to the one from Remark \ref{noImprovement} show that the bound $O(n^{3/2})$ cannot be improved unless we impose additional constraints on $n$.

When $\baseField=\RR$, additional tools from incidence geometry become available, and we can say more.
\begin{thm}\label{spheresInR3Thm}
Let $\mathcal{S}$ a set of $n$ oriented spheres in $\RR^3$. Then for each $3\leq k\leq n$, $\mathcal{S}$ determines $O(n^{3/2}k^{-5/2}+nk^{-1})$ $k$-rich pencils of contacting spheres.
\end{thm}
Note that a $k$-rich pencil determines $\binom{k}{2}$ pairs of contacting spheres. Since the quantity $\binom{k}{2} n^{3/2}k^{-5/2}$ is dyadically summable in $k$, Theorem \ref{spheresInR3Thm} allows us to bound the number of pairs of contacting spheres, provided not too many spheres are contained in a common pencil or pair of complimentary conic sections.

\begin{cor}\label{numberOfSphereContacts}
Let $\mathcal{S}$ a set of $n$ oriented spheres in $\RR^3$. Suppose that no pencil of contacting spheres is $\sqrt n$ rich. Then at least one of the following two things must occur
\begin{itemize}
\item There is a pair of complimentary conic sections $\mathfrak{C},\mathfrak{C}^\prime$ so that 
\begin{equation}\label{pairOfCompConicSectoins}
|\mathfrak{C}\cap \mathcal{S}|\geq \sqrt n\quad\textrm{and}\quad |\mathfrak{C}^\prime\cap \mathcal{S}|\geq \sqrt n.
\end{equation}
\item There are $O(n^{3/2})$ pairs of contacting spheres.
\end{itemize}
\end{cor}

The following (rather uninteresting) example shows that Theorem \ref{spheresInR3Thm} can be sharp when there are pairs of complimentary conic sections that contain almost $\sqrt n$ spheres from $\mathcal{S}$.
\begin{example}
Let $n = 2m^2$ and let $\mathcal{S}$ be a disjoint union $\mathcal{S} = \bigsqcup_{j=1}^m(\mathcal{S}_j\cup \mathcal{S}_j^\prime)$, where for each index $j$, each of $\mathcal{S}_j$ and $\mathcal{S}_j^\prime$ is a set of $m$ spheres contained in complimentary conic sections. The spheres in $\mathcal{S}$ determine $m^3=(n/2)^{3/2}$ $2$-rich pencils, and no pair of complimentary conic sections satisfy \eqref{pairOfCompConicSectoins}. 
\end{example} 
A similar example shows that Theorem \ref{spheresInR3Thm} can be sharp when there are pencils that contain almost $\sqrt n$ spheres from $\mathcal{S}$. More interestingly, the following ``grid'' construction shows that Theorem \ref{spheresInR3Thm} can be sharp even when $\mathcal{S}$ does not contain many spheres in a pencil or many spheres in complimentary conic sections. 
\begin{example}\label{gridExample}
Let $n = m^4$, let $[m]=\{1,\ldots,m\}$, and let $\mathcal{S}$ consist of all oriented spheres centered at $(a,b,c)$ of radius $d$, where $a,b,c,d \in [m]$. We can verify that the equation
\begin{equation}\label{tangencyIntegerPoints}
(a-a^\prime)^2 + (b-b^\prime)^2 + (c-c^\prime)^2 = (d-d^\prime)^2,\ a,a^\prime,b,b^\prime,c,c^\prime,d,d^\prime\in[m]
\end{equation}
has roughly $m^6 = n^{3/2}$ solutions. No linear pencil of contacting spheres contains more than $m = n^{1/4}$ spheres from $\mathcal{S}$. This means that $\mathcal{S}$ determines roughly $n^{3/2}$ 2-rich pencils, and also determines roughly $n^{3/2}$ pairs of contacting spheres. On the other hand, if we fix three non-collinear points $(a_j,b_j,c_j,d_j)\in[m]^4$, $j=1,2,3$, then the set of points $(a^\prime,b^\prime,c^\prime,d^\prime)\in[m]^4$ satisfying \eqref{tangencyIntegerPoints} for each index $j=1,2,3$ must be contained in an irreducible degree two curve in $\RR^4$. Bombieri and Pila \cite{BP} showed that such a curve contains $O(n^{1/8+\eps})$ points from $[m]^4$. We conclude that every pair of complimentary conic sections contains $O(n^{1/8+\eps})$ spheres from $\mathcal{S}$.
\end{example}
\subsection{Thanks}
The author would like to thank Kevin Hughes and Jozsef Solymosi for helpful discussions, and Adam Sheffer for comments and corrections to an earlier version of this manuscript. The author would also like to thank Gilles Castel for assistance creating Figures \ref{pencilPic} and \ref{complimentaryConicSectionsPic}. The author was partially funded by a NSERC discovery grant.

\section{Lie's line-sphere correspondence}\label{lieLineSphereSec}
In this section we will discuss the contact-incidence isomorphism introduced in Section \ref{introSection}. Throughout this section, $\baseField$ will be a field in which $-1$ is not a square, and $\extendedField=\baseField[i]$ is a degree-two extension of $\baseField$, where $i^2=-1.$
\subsection{Lines and the Klein quadric}
In this section we will be concerned with points in $\extendedField^6$ and its projectivization $\LP^5$. We will write elements of $\extendedField^6$ using the index set $(p_{01},p_{02},p_{03},p_{23},p_{31},p_{12})$, and elements of $\LP^5$ will be denoted $[p_{01}:p_{02}:p_{03}:p_{23}:p_{31}:p_{12}]$. Let $\klein(\cdot,\cdot) $ be the symmetric bilinear form on $\extendedField^6$ given by
\begin{equation*}
\begin{split}
\klein\big((p_{01},p_{02},p_{03},p_{23},p_{31},p_{12}), &(p_{01}^\prime,p_{02}^\prime,p_{03}^\prime,p_{23}^\prime,p_{31}^\prime,p_{12}^\prime)\big) \\
&= 
p_{01}p_{23}^\prime + p_{23}p_{01}^\prime
+ p_{02}p_{31}^\prime + p_{31}p_{02}^\prime
+ p_{03}p_{12}^\prime + p_{12}p_{03}^\prime.
\end{split}
\end{equation*}
We define the Klein quadric to be the set
\begin{equation*}
\begin{split}
\{ \mathbf p = [p_{01}:p_{02}:p_{03}:p_{23}:p_{31}:p_{12}]\in\LP^5 \colon p_{01}p_{23}+p_{02}p_{31}+p_{03}p_{12}=0
\}.
\end{split}
\end{equation*}
Since the polynomial $p_{01}p_{23}+p_{02}p_{31}+p_{03}p_{12}$ is homogeneous, the above set is well-defined. Since $\operatorname{char}(\baseField)\neq 2$, the above relation can also be written as
\begin{equation}
\klein(\mathbf p,\mathbf p)  = 0.
\end{equation}
We will call this the Pl\"ucker relation.

There is a one-to-one correspondence between projective lines in $\LP^3$ and points in the Klein quadric. For our purposes, however, it will be more useful for us to identify a certain class of (affine) lines in $\extendedField^3$ with a large subset of the Klein quadric. Concretely, the (affine) line $(0,s,t) + \extendedField(1,u,v)$ can be identified with the point
\begin{equation}\label{pluckerCoordinate}
\begin{array}{rcccccccccccl}
[ & p_{01}&:&p_{02}&:&p_{03}&:&p_{23}&:&p_{31}&:&p_{12}&]\\
=[&1&:&u&:&v&:&sv-tu&:&t&:&-s&].
\end{array}
\end{equation}
We will call this point the Pl\"ucker coordinates of the line. Conversely, the point with Pl\"ucker coordinates $[1:p_{02}:p_{03}:p_{23}:p_{31}:p_{12}]$ can be identified with the line $(0,-p_{12},p_{31})+\extendedField(1,p_{02},p_{03})$.

Two lines $\ell$ and $\tilde\ell$ are coplanar (and thus either intersect or are parallel) if and only if their respective Pl\"ucker points $\mathbf p$ and $\mathbf p^\prime$ satisfy the relation $\klein(\mathbf p,\mathbf p^\prime)  =0$.

\subsection{Oriented spheres and the Lie quadric}
In this section we will recall some basic facts from Lie sphere geometry. The primary reference is \cite{Cecil}, especially Chapter 2, and \cite{Milson}. Lie sphere geometry studies objects called ``Lie spheres,'' which unify the notion of a sphere, point, and plane (the latter two objects can be thought of as spheres that have zero and infinite radius, respectively). We will restrict attention to points and spheres.

Let $\lie(\cdot,\cdot) $ be the symmetric bilinear form on $\baseField^6$ defined by
$$
\lie \big((a,b,c,d,e,f), (a^\prime,b^\prime,c^\prime,d^\prime,e^\prime,f^\prime)\big) = 
2bb^\prime
+ 2cc^\prime
+ 2dd^\prime
-ae^\prime - ea^\prime
-2ff^\prime.
$$
We define the Lie quadric to be the set

\begin{equation*}
\{\mathbf q = [a:b:c:d:e:f]\in \KP^5\colon \lie (\mathbf q,\mathbf q)=0\}.
\end{equation*}
Since the polynomial $Q_{\lie}(\mathbf q) = \lie (\mathbf q,\mathbf q)$ is homogeneous, the above set is well-defined. We will refer to the equation $\lie (\mathbf q,\mathbf q)=0$ as as the Lie relation. It can also be written as 
\begin{equation}
b^2+c^2+d^2-ae-f^2=0.
\end{equation}
For each point $\mathbf q\in\KP^5$ with $\lie (\mathbf q,\mathbf q)=0$, we define the set
\begin{equation}\label{relationSphere}
S_{\mathbf q} = \{(x,y,z)\in \baseField^3 \colon a(x^2+y^2+z^2)+2bx + 2cy + 2dz + e = 0\}.
\end{equation}
When $a\neq 0$, $S_{\mathbf q}$ is the sphere defined by the equation
\begin{equation}\label{eqnOfSphere}
(x+b/a)^2 + (y+c/a)^2 + (z+d/a)^2 = (f/a)^2.
\end{equation}
In particular, the oriented sphere $S$ centered at the point $(x_1,y_1,z_1)$ with radius $r_1$ can be identified with the point
\begin{equation}\label{LieCoordinate}
\begin{array}{rcccccccccccl}
\mathbf{q}_S = [ & a &:&b&:&c&:&d&:&e&:&f&]\\
=[&1&:&-x_1& :&-y_1&:&  -z_0&:&x_1^2+y_1^2+z_1^2-r^2_1&:&r_1&].
\end{array}
\end{equation}
%
%

A set of the form $S_{\mathbf q}$, with $\mathbf q$ in the Lie Quadric will be referred to as a Lie sphere. We say that two Lie-spheres $S$ and $S^\prime$ are in ``contact'' if their corresponding Lie points $\mathbf q$ and $\mathbf q^\prime$ satisfy the relation $\lie(\mathbf q, \mathbf q^\prime)=0$, or equivalently
\begin{equation}\label{orientedTangencyRelationTwo}
(b- b^\prime)^2 + (c- c^\prime)^2 + (d- d^\prime)^2 = (f- f^\prime)^2.
\end{equation}
Indeed, examining \eqref{defnOfContact}, \eqref{LieCoordinate}, and \eqref{orientedTangencyRelationTwo}, we see that two oriented spheres $S$ and $S^\prime$ are in contact in the sense of \eqref{defnOfContact} precisely when their corresponding Lie points satisfy $\lie(\mathbf q, \mathbf q^\prime)=0$. 

\subsection{Line-sphere correspondence}\label{lineSphereCorrespondenceSubsection}
Consider the following map $\phi$ from the Lie quadric (which is a subset of $\KP^5$) to the Klein quadric (which is a subset of $\LP^5$). The point $\mathbf q = [a:b:c:d:e:f]$ is mapped to
\begin{equation}\label{phiDefn}
\begin{array}{rcccccccccccl}
\phi(\mathbf q) = [ & p_{01}&:&p_{02}&:&p_{03}&:&p_{23}&:&p_{31}&:&p_{12}&]\\
=[&a&:& d+f &:& -b + ic&:& -e &:&d-f&:&-b-ic&].
\end{array}
\end{equation}
If $\mathbf q,\mathbf q^\prime\in \KP^5$, then 
\begin{equation}
\begin{split}
\klein(\phi(\mathbf q),\phi(\mathbf q^\prime)) & = p_{01} p_{23}^\prime + p_{23}   p_{01}^\prime 
+ p_{02} p_{31}^\prime + p_{31} p_{02}^\prime 
+ p_{03} p_{12}^\prime + p_{12} p_{03}^\prime\\
&= (a)(- e^\prime) + (-e) ( a^\prime)
+ (d+f)( d^\prime -  f^\prime) + (d-f)( d^\prime+ f^\prime)\\
&\qquad\qquad\qquad\qquad\qquad\qquad
+ (-b + ic)(- b^\prime-i c^\prime) + (-b-ic)(- b^\prime + i c^\prime)\\
&= -a e^\prime - e a^\prime + 2d d^\prime -2f f^\prime +2b b^\prime+2c c^\prime\\
& = \lie(\mathbf q,\mathbf q^\prime).
\end{split}
\end{equation}
Setting $\mathbf q = \mathbf q^\prime$, we see that if $\mathbf q$ is an element of the Lie Quadric then $\phi(\mathbf q)$ is an element of the Klein Quadric. Furthermore, if $\mathbf q$ and $\mathbf q^\prime$ are distinct elements in the Lie Quadric, then the Lie spheres corresponding to $\mathbf q$ and $\mathbf q^\prime$ are in contact if and only if their images under $\phi$ are coplanar.

If $S$ is the oriented sphere centered at the point $(x_1,y_1,z_1)\in \baseField^3$ with radius $r_1\in \baseField$, then by combining \eqref{LieCoordinate} and \eqref{phiDefn}, we see that $S$ is mapped to the line in $\extendedField^3$ with Pl\"ucker coordinates 

\begin{equation}\label{imageOfSphere}
\begin{array}{rcccccccccccl}
[ & p_{01}&:&p_{02}&:&p_{03}&:&p_{23}&:&p_{31}&:&p_{12}&]\\
=[&a&:& d+f &:& -b + ic&:& -e &:&d-f&:&-b-ic&]\\
=[&1&:& -z_1+r_1 &:& x_1 - iy_1&:& r_1^2-x_1^2-y_1^2-z_1^2 &:&-z_1-r_1&:&x_1 +iy_1&].
\end{array}
\end{equation}
This corresponds to the line  $(0,\omega,t)+\extendedField(1,u,\bar\omega)$, where $t = -z_1-r_1$, $u=-z_1+r_1$, and $\omega = -x_1-iy_1$. As we saw in Section \ref{introSection}, lines of this form are contained in the Heisenberg group $\mathbb{H}$.

\subsection{Pencils of contacting spheres}\label{pencilSection}
In this section we will explore the structure of pencils of contacting spheres in $\baseField^3$, the Lie quadric, and the Heisenberg group. We will begin by proving Lemma \ref{linearPencilProp}. For the reader's convenience we will restate it here.

\begin{linearPencilLem}
Let $S_1$ and $S_2$ be distinct oriented spheres that are in contact. Then the set of oriented spheres that are in contact with $S_1$ and $S_2$ is a pencil of contacting spheres.
\end{linearPencilLem}
\begin{proof}
We will prove the following equivalent statement: Let $\ell_1$ and $\ell_2$ be coplanar lines in the Heisenberg group that are not parallel to the $xy$ plane. Let $\mathcal{L}$ be the set of lines in the Heisenberg group not parallel to the $xy$ plane that are coplanar with both $\ell_1$ and $\ell_2$ (so in particular $\ell_1,\ell_2\in\mathcal{L})$. Then all of the lines in $\mathcal{L}$ are contained in a common plane in $\extendedField^3$ and all pass through a common point (possibly at infinity). Furthermore, $\mathcal{L}$ is maximal in the sense that no additional lines can be added to $\mathcal{L}$ while maintaining the property that each pair of lines is coplanar.

First, observe that every line in $\mathbb{H}$ that is not parallel to the $xy$ plane points in a direction $\mathbf v$ of the form $\mathbf v=(1, b, c-di)\in \extendedField^3$, with $b,c,d\in \baseField$. Every such line that points in the direction $\mathbf v$ must be contained in the plane 
\begin{equation}\label{parallelVPlane}
\Pi_{\mathbf v}=\{(x,y,z)\in \extendedField^3\colon\ bx + y = c+di\}.
\end{equation}
In particular, this implies that any set of lines in $\mathbb{H}$ that all point in the same direction and are not parallel to the $xy$ plane must be contained in a common plane in $\extendedField^3$.

Next, observe that if $\mathbf w = (x_1,y_1,z_1) \in\mathbb{H}$, then every line contained in $\mathbb{H}$ passing through $w$ must be contained in the plane
\begin{equation}\label{tangentPlaneH}
\Pi_{\mathbf w}=\{(x,y,z)\in \extendedField^3\colon\ iy_1 x -ix_1 y - iz = -iz_1\}.
\end{equation}
If $\mathbf w$ and  $\mathbf w^\prime$ are distinct points in $\mathbb{H}$, then the corresponding planes $\Pi_{\mathbf w}$ and $\Pi_{\mathbf w^\prime}$ are also distinct, and thus either are disjoint or intersect in a line. In particular, this implies that if two distinct lines contained in $\mathbb{H}$ intersect at a point $w\in\mathbb{H}$, then the unique plane in $\extendedField^3$ containing both lines is precisely $\Pi_{\mathbf w}$. 

Now, let $\ell_1,\ell_2,$ and $\ell_3$ be three lines in $\mathbb{H}$ that are pairwise coplanar. The following argument will show that $\ell_1,\ell_2,$ and $\ell_3$ must be contained in a common plane and must pass through a common point (possibly at infinity).

First, if all three lines are parallel, then we have already shown that they are contained in a common plane and we are done. If not all three lines are parallel, then we can suppose that $\ell_1$ and $\ell_2$ intersect at a point $\mathbf w\in\mathbb{H}$, so $\ell_1$ and $\ell_2$ must be contained in $\Pi_{\mathbf w}$. If $\ell_3$ is parallel to one of these lines, then without loss of generality we can suppose it is parallel to $\ell_1$, and both lines point in direction $\mathbf v$. Then $\ell_2$ intersects two distinct lines contained in $\Pi_{\mathbf v}$, so $\ell_2$ must also be contained in $\Pi_{\mathbf v}$. But if $\mathbf w^\prime = \ell_2\cap\ell_3$, then this implies $\Pi_{\mathbf w} = \Pi_{\mathbf w^\prime}=\Pi_{\mathbf v}$, which is impossible since $\mathbf w\neq \mathbf w^\prime$. Thus we may suppose that $\ell_3$ is not parallel to either $\ell_1$ or $\ell_2$. Suppose $\ell_3$ intersects $\ell_1$ and $\ell_2$ at distinct points. Then $\ell_3$ must also be contained in $\Pi_{\mathbf w}$. But at least one of the points $\ell_1\cap\ell_2$ and $\ell_1\cap\ell_3$ must differ from $\mathbf w$, which implies there is a point $\mathbf w^\prime\neq \mathbf w$ with $\Pi_{\mathbf w}=\Pi_{\mathbf w^{\prime}}$; this is a contradiction. We conclude that $\ell_3$ passes through $\ell_1\cap\ell_2$. Since all lines in $\mathbb{H}$ passing through $\mathbf w$ must be contained in $\Pi_{\mathbf w}$, we conclude that $\ell_1,\ell_2,$ and $\ell_3$ are coplanar and coincident.

It now immediately follows that if every line in $\mathcal{L}$ is coplanar with both $\ell_1$ and $\ell_2$, then each of these lines must either be parallel with $\ell_1$ and $\ell_2$ (if $\ell_1$ and $\ell_2$ are parallel), or must pass through their common intersection point (if $\ell_1$ and $\ell_2$ are not parallel). Furthermore, each line in $\mathcal{L}$ must be contained in the plane spanned by $\ell_1$ and $\ell_2$. In particular, all of the lines in $\mathcal{L}$ are coplanar and pass through a common point (possibly at infinity). By construction the set $\mathcal{L}$ is maximal, i.e.~no additional lines can be added to $\mathcal{L}$ while maintaining the property that all pairs of lines in $\mathcal{L}$ are coplanar.
\end{proof}

If $\mathbf w= (x_1,y_1,z_1)\in\mathbb{H}$, then the set of lines in $\mathbb{H}$ that are not parallel to the $xy$ plane containing $\mathbf w$ are of the form $(0,c+id,a)+\extendedField(1,b,c-id)$, where
\begin{equation}\label{linesThruAPt}
\begin{split}
c+b\operatorname{Re} x_1 &  = \operatorname{Re} y_1,\\
d+b\operatorname{Im} x_1 &  = \operatorname{Im} y_1,\\
a + c \operatorname{Re} x_1 + d \operatorname{Im}x_1 &  = \operatorname{Re} z_1.
\end{split}
\end{equation}
The set of quadruples $(a,b,c,d)$ satisfying \eqref{linesThruAPt} belong to a line in $\baseField^4$; we will call this line $V_{\mathbf w}$. Similarly, if $\mathbf v=(1, b, c-di)\in \extendedField^3$, with $b,c,d\in \baseField$, the set of lines in $\mathbb{H}$ pointing in direction $\mathbf v$ are of the form $(0,c+id,a)+\extendedField(1,b,c-id)$ with $a\in \baseField$. We will call this line $V_{\mathbf v}$.
\begin{rem}\label{allSameDifferentRadiiRem}
The above discussion implies that the spheres in a pencil of contacting spheres must either all have the same (signed) radius, or must all have different radii. 
\end{rem}

\subsection{Complimentary conics}\label{conicsSection}
In this section we will discuss the structure of sets of spheres that form complimentary conics. The next lemma says that pencils of spheres and and pairs of complimentary conics are the only configurations allowing many spheres to be in contact.

\begin{lem}\label{doublyRuledSpheresProp}
Let $\mathcal{S}$ and $\mathcal{S}^\prime$ be sets of oriented spheres, each of cardinality at least three, so that every sphere in $\mathcal{S}$ is in contact with every sphere in $\mathcal{S}^\prime$ and vice-versa. Then either $\mathcal{S} \cup \mathcal{S}^\prime$ is contained in a pencil of contacting spheres (in which case every sphere in $\mathcal{S} \cup \mathcal{S}^\prime$ is in contact with every other sphere), or $\mathcal{S}$ and $\mathcal{S}^\prime$ are contained in complimentary conic sections.
\end{lem}
\begin{proof}
First, suppose that two spheres $S_1,S_2 \in \mathcal{S}$ are in contact. Then by Proposition \ref{linearPencilProp}, the spheres in $\{S_1\cup S_2\}\cup \mathcal{S}^\prime$ are contained in a pencil $\mathfrak{P}$ of contacting spheres. Since $|\mathcal{S}|\geq 3$, we have that $|\mathfrak{P}^\prime\cap \mathcal{S}^\prime|\geq 3$. Proposition \ref{linearPencilProp} now implies that every sphere from $\mathcal{S}$ is contained in $\mathfrak{P}$. We conclude that $\mathcal{S} \cup \mathcal{S}^\prime$ is contained in a pencil of contacting spheres. An identical argument applies if two spheres $S_1^\prime,S_2^\prime \in \mathcal{S}^\prime$ are in contact. Thus we can suppose that no two spheres from $\mathcal{S}$ are in contact, and no two spheres from $\mathcal{S}^\prime$ are in contact. Let $\mathfrak{C}^\prime$ be the set of spheres in contact with each sphere from $\mathcal{S}$, and let $\mathfrak{C}$ be the set of spheres that are in contact with each sphere from $\mathfrak{C}^\prime$; we have that $\mathfrak{C}$ and $\mathfrak{C}^\prime$ are complimentary conics containing $S$ and $S^\prime$, respectively.
\end{proof}

We will now consider the structure of a pair of complimentary conic sections. Let $\mathbf{q}_1,\mathbf{q}_2,\mathbf{q}_3$ be elements of the Lie quadric, no two of which are in contact. Let $\ell_1,\ell_2$ and $\ell_3$ be the lines in $\extendedField^3$ corresponding to the images of $\mathbf{q}_1,\mathbf{q}_2,$ and $\mathbf{q}_3$ under $\phi$. Let $\overline \baseField$ be the algebraic closure of $\baseField$ (so in particular $E\subset \overline \baseField$), and let $\overline{\ell_i}$ be the Zariski closure of $\ell_i$ in $\overline \baseField^3$. Then $\overline{\ell_1},\overline{\ell_2}$ and $\overline{\ell_3}$ are skew lines in $\overline \baseField^3$, and the set of lines in $\overline \baseField^3$ that are coplanar with each $\overline{\ell_i}$ form a ruling of a doubly ruled surface in $\overline \baseField^3$. Let $\mathfrak{R}$ be the set of lines in the ruling that contains the lines $\overline{\ell_1},\overline{\ell_2}$ and $\overline{\ell_3}$ and let $\mathfrak{R}^\prime$ be the set of lines in the other ruling. The sets $\mathfrak{R}$ and $\mathfrak{R}^\prime$ are irreducible conic curves in the variety $\{\mathbf{p} \in \overline \baseField \mathbf{P}^5\colon \klein(\mathbf p,\mathbf p)=0\}$. We then have $\mathfrak{C} = \phi^{-1}(\mathfrak{R})$ and $\mathfrak{C}^\prime = \phi^{-1}(\mathfrak{R}^\prime)$. Since $\phi$ is linear, this implies that $\mathfrak{C}$ and $\mathfrak{C}^\prime $ are precisely the $\baseField$-points of an irreducible curve in $\overline \baseField \mathbf{P}^5$ that is contained in the Lie quadric (or more precisely, contained in the variety $\{\mathbf{q} \in \overline \baseField \mathbf{P}^5\colon \lie(\mathbf q,\mathbf q)=0\}$). Recalling the identification \eqref{eqnOfSphere} of points in the Lie quadric with spheres in $\baseField^3$, we see that $\mathfrak{C}$ and $\mathfrak{C}^\prime$ correspond to sets of oriented spheres whose $(x,y,z,r)$ coordinates are contained in conic sections, neither of which are lines. 

Note that for some triples $\mathbf{q}_1,\mathbf{q}_2,\mathbf{q}_3$ it is possible that no elements of the Lie quadric will be in contact with each $\mathbf{q}_i$. This can happen, for example, if $\baseField=\RR$; $\mathbf{q}_1$ and $\mathbf{q}_2$ correspond to disjoint spheres; and $\mathbf{q}_3$ corresponds to a sphere contained inside $\mathbf{q}_1$. In this situation the complimentary conic sections $\mathfrak{C}$ and $\mathfrak{C}^\prime$ are well defined, but $\mathfrak{C}^\prime$ does not have any $\RR$-points.

\section{Incidence geometry in the Heisenberg group}\label{incidenceGeomHeisenberg}

In this section we will explore the incidence geometry of lines in the Heisenberg group. In particular, we will prove Theorems \ref{spheresInGeneralFieldThm}, \ref{unitDistancesR3}, and \ref{pointSphereThm}. Our main tool is the following structure theorem for sets of lines in three space that determine many 2-rich points. The version stated here is Theorem 3.8 from \cite{GZ}; a similar statement can also be found in \cite{Kollar}.

\begin{thm}\label{2RichManyLinesInPlaneOrRegulus}
There are absolute constants $c>0$ and $C$ so that the following holds. Let $\mathcal{L}$ be a set of $n$ lines in $\extendedField^3$, with $n\leq c(\operatorname{char} \extendedField)^2$ (if $\extendedField$ has characteristic zero then we impose no constraints on $n$). Then for each $A\geq C\sqrt n$, either $\mathcal{L}$ determines at most $CAn$ $2$-rich points, or there is a plane or doubly-ruled surface in $\extendedField^3$ that contains at least $A$ lines from $\mathcal{L}$.
\end{thm}

Recall that lines contained in the Heisenberg group are somewhat special. In particular, Lemma \ref{linearPencilProp} says that any set of pairwise coplanar lines must also pass through a common point (possibly at infinity). The next lemma is a version of Theorem \ref{2RichManyLinesInPlaneOrRegulus} adapted to lines in the Heisenberg group. The hypotheses and conclusions of the theorem have also been slightly tweaked to better fit our needs for proving Theorems \ref{spheresInGeneralFieldThm}, \ref{unitDistancesR3}, and \ref{pointSphereThm}.
\begin{lem}\label{Coloured2RichPointsInHeisenberg}
There is an absolute constant $C_1$ so that the following holds. Let $\mathcal{L}$ be a set of $n$ lines in the Heisenberg group that are not parallel to the $xy$ plane, with $n\leq (\operatorname{char} \extendedField)^2$ (if $\extendedField$ has characteristic zero then we impose no constraints on $n$). Suppose that each line in $\mathcal{L}$ is coloured either red or blue. Then for each $A\geq C_1\sqrt n$, either there are at most $C_1An$ points that are incident to at least one red and one blue line, or there is a doubly-ruled surface with one ruling that contains at least $A$ red lines and a second ruling that contains at least $A$ blue lines. 
\end{lem}
\begin{proof}
Suppose that there are more than $C_1An$ points that are incident to at least one red and one blue line (we will call such points bichromatic 2-rich points). We will show that if $C_1$ is chosen sufficiently large then there is a doubly-ruled surface with one ruling that contains at least $A$ red lines and a second ruling that contains at least $A$ blue lines. 

Observe that Lemma \ref{Coloured2RichPointsInHeisenberg} assumes that $n\leq (\operatorname{char} \extendedField)^2$, while Theorem \ref{2RichManyLinesInPlaneOrRegulus} places the more stringent requirement $n\leq c(\operatorname{char} \extendedField)^2$.  Our first task will be to reduce the size of $\mathcal{L}$ slightly. Without loss of generality we can assume that $0<c\leq 1$, since if $c>1$ then we can replace $c$ with $1$ and Theorem \ref{2RichManyLinesInPlaneOrRegulus} remains true. Let $\mathcal{L}^\prime\subset\mathcal{L}$ be a set obtained by randomly keeping each element in $\mathcal{L}$ with probability $c/10$. With high probability, a set of this form will have cardinality at most $c(\operatorname{char} \extendedField)^2$ and will will determine at least $C_2An$ bichromatic 2-rich points, with $C_2=\frac{C_1c^2}{1000}$. In particular, we can assume that $\mathcal{L}^\prime$ has both of these properties.

Our next task is to prune the set $\mathcal{L}^\prime$ slightly so that all of the unpruned lines contain many bichromatic 2-rich points. Define $\mathcal{L}_0^\prime=\mathcal{L}^\prime$. For each index $j=0,1,\ldots$, let $\mathcal{L}_{j+1}^\prime$ be obtained by removing a line from $\mathcal{L}_j^\prime$ that contains at most $2A$ bichromatic 2-rich points. If no such line exists, then halt. Observe that $\mathcal{L}_j^\prime$ has cardinality $|\mathcal{L}^\prime|-j$ and determines at least $C_2An -2Aj$ bichromatic 2-rich points. If $C_2$ is sufficiently large then this process must halt for some index $j$ with $|\mathcal{L}_j^\prime|>0$. Let $\mathcal{L}^{\prime\prime}$ be the resulting set; we have that $\mathcal{L}^{\prime\prime}$ determines at least $C_3An$ bichromatic 2-rich points, with $C_3 = C_2/2$, and every line determines at least $2A$ bichromatic 2-rich points.

Apply Theorem \ref{2RichManyLinesInPlaneOrRegulus} to $\mathcal{L}^{\prime\prime}$. If $C_1$ (and thus $C_3$) is chosen sufficiently large then there exists a plane or doubly-ruled surface $Z$ that contains at least $2A$ lines from $\mathcal{L}^{\prime\prime}$; Let $\mathcal{L}^{\prime\prime}_Z$ be the set of lines from $\mathcal{L}^{\prime\prime}$ that are contained in $Z$. We claim that $Z$ cannot be a plane; indeed if $Z$ was a plane then by Lemma \ref{linearPencilProp} the lines in $\mathcal{L}^{\prime\prime}_Z$ must all intersect at a common point (possibly at infinity). In particular, each of the lines in $\mathcal{L}^{\prime\prime}_Z$ contain at least $2A$ 2-rich points, and at least $2A-1$ of these 2-rich points must come from lines in $\mathcal{L}^{\prime\prime}\backslash \mathcal{L}^{\prime\prime}_Z$. Thus the lines in $\mathcal{L}^{\prime\prime}\backslash \mathcal{L}^{\prime\prime}_Z$ must intersect $Z$ in at least $(2A)(2A-1)>n$ distinct points. Since each line in $\mathcal{L}^{\prime\prime}\backslash \mathcal{L}^{\prime\prime}_Z$ intersects $Z$ in at most one point, this is impossible. 

We conclude that $Z$ is a doubly ruled surface. Since each line in $\mathcal{L}^{\prime\prime}$ that is not contained in $Z$ can intersect $Z$ in at most two points, by pigeonholing there exists a line $\ell\in \mathcal{L}^{\prime\prime}_Z$ that contains $\geq 2A$ bichromatic 2-rich points, and at least $A$ of these points come from lines contained in the dual ruling of $Z$. Without loss of generality the line $\ell$ is red; this means that the ruling dual to $\ell$ contains at least $A$ blue lines. Again by pigeonholing, at least one of these blue lines $\ell^\prime$ contains $\geq 2A$ bichromatic 2-rich points, and at least $A$ of these points come from lines contained in the ruling of $Z$ dual to $\ell^\prime$ (i.e. the ruling that contains $\ell$). We conclude that the ruling containing $\ell$ must contain at least $A$ red lines.
\end{proof}
\begin{cor}\label{2RichPointsInHeisenberg}
Let $\mathcal{L}$ be a set of $n$ lines in the Heisenberg group that are not parallel to the $xy$ plane, with $n\leq (\operatorname{char} \extendedField)^2$ (if $\extendedField$ has characteristic zero then we impose no constraints on $n$). Then either $\mathcal{L}$ determines $O(n^{3/2})$ 2-rich points, or there is a doubly-ruled surface, each of whose rulings contain at least $\sqrt n$ lines from $\mathcal{L}$.
\end{cor}
\begin{proof}
Suppose that $\mathcal{L}$ determines $C_2n^{3/2}$ 2-rich points. Randomly colour each line in $\mathcal{L}$ either red or blue. With high probability $\mathcal{L}$ determines at least $\frac{C_2}{3}n^{3/2}$ bichromatic 2-rich points. If $C_2$ is sufficiently large, then Lemma \ref{2RichPointsInHeisenberg} implies that there is a doubly-ruled surface, each of whose rulings contain at least $ \sqrt n$ lines from $\mathcal{L}$.
\end{proof}

Lemma \ref{Coloured2RichPointsInHeisenberg} allows us to understand configurations of lines in $\extendedField^3$ that contain many bichromatic 2-rich points. The next lemma concerns configurations of lines in $\extendedField^3$ that contain many $k$-rich points for $k\geq 3$. Since the proof of the next lemma is very similar to that of Lemma \ref{Coloured2RichPointsInHeisenberg} (in particular it uses the same ideas of random sampling and refinement), we will just provide a brief sketch that highlights where the proofs differ.

\begin{lem}\label{3RichPointsInHeisenberg}
Let $\mathcal{L}$ be a set of $n$ lines in the Heisenberg group that are not parallel to the $xy$ plane, with $n\leq (\operatorname{char} \extendedField)^2$ (if $\extendedField$ has characteristic zero then we impose no constraints on $n$). Let $k\geq 3$. Then $\mathcal{L}$ determines $O(n^{3/2}k^{-3/2})$ $k$-rich points.
\end{lem}
\begin{proof}[Proof sketch]
Let $\mathcal{L}^\prime\subset\mathcal{L}$ be obtained by randomly selecting each line $\ell\in\mathcal{L}$ with probability $2/k$. Then with positive probability we have that $|\mathcal{L}^\prime|\leq 100|\mathcal{L}|/k$, and the number of 3-rich points determined by $\mathcal{L}^\prime$ is at least half the number of $k$-rich points determined by $|\mathcal{L}|$. We now argue as in the proof of Corollary \ref{2RichPointsInHeisenberg}; if $\mathcal{L}^\prime$ determines $\geq C_1 (n/k)^{3/2}$ 3-rich points, then there must exist a doubly-ruled surface $Z$ that contains $\geq (n/k)^{1/2}$ lines from $\mathcal{L}^\prime$, and each of these lines must contain at least $3\sqrt n$ 3-rich points. In particular, there must be at least $3n$ 3-rich points contained in $Z$. Since $Z$ is doubly (not triply!) ruled, this means that each of these 3-rich points must be incident to a line from $\mathcal{L}^\prime$ that is not contained in $Z$. Since each line in $\mathcal{L}^\prime$ not contained in $Z$ can intersect $Z$ in at most 2 points, the number of 3-rich points created in this way is at most $2n$, which is a contradiction.
\end{proof}

With these preliminary results established, we are now ready to prove the main results of this section. 
\begin{proof}[Proof of Theorem \ref{spheresInGeneralFieldThm}]
First, using the line-sphere correspondence described in Section \ref{lineSphereCorrespondenceSubsection}, Theorem \ref{spheresInGeneralFieldThm} is implied by the following statement about lines in the Heisenberg group:

Let $\mathcal{L}$ be a set of $n$ lines in the Heisenberg group that are not parallel to the $xy$ plane, with $n\leq (\operatorname{char} \extendedField)^2$ (if $\extendedField$ has characteristic zero then we impose no constraints on $n$). Then for each $3\leq k\leq n$, there are $O(n^{3/2}k^{-3/2})$ $k$-rich points. Furthermore, at least one of the following two things must occur:
\begin{itemize}
\item There is a doubly-ruled surface, each of whose ruling contain at least $\sqrt n$ lines from $\mathcal{L}$. 
\item $\mathcal{L}$ determines $O(n^{3/2})$ $2$-rich points.
\end{itemize}
The theorem now follows immediately from Corollary \ref{2RichPointsInHeisenberg} and Lemma \ref{3RichPointsInHeisenberg}. 
\end{proof}

Our proof of Theorem \ref{unitDistancesR3} will follow a similar strategy to the proof of Theorem \ref{spheresInGeneralFieldThm}. The main thing to verify is that not too many lines can be contained in a doubly-ruled surface.

\begin{proof}[Proof of Theorem \ref{unitDistancesR3}]
Let $(x_i,y_i,z_i)\in \baseField^3,\ i=1,2,3$ be three points. 
Consider the set of points $(x,y,z)\in \baseField^3$ satisfying 
\begin{align}
(x-x_1)^2 + (y-y_1)^2 + (z-z_1)^2 & = r^2\label{firstSphereEqn},\\
(x-x_2)^2 + (y-y_2)^2 + (z-z_2)^2 & = r^2\label{secondSphereEqn},\\
(x-x_3)^2 + (y-y_3)^2 + (z-z_3)^2 & = r^2\label{thirdSphereEqn}.
\end{align}
Note that these three equations are satisfied precisely when 
$$(x_i,y_i,z_i)\in S(x,y,z),\quad i = 1,2,3,$$
where
\begin{equation}\label{defnSxyzr}
S(x,y,z) = \{(x^\prime,y^\prime,z^\prime)\in \baseField^3\colon (x-x^\prime)^2 + (y-y^\prime)^2 + (z-z^\prime)^2 = r^2 \}.
\end{equation}
In particular, since $-1$ is not a square in $\baseField$, the set $S(x,y,z)$ does not contain any lines (see e.g. \cite[Lemma 6]{Rudnev2}), so \eqref{firstSphereEqn}, \eqref{secondSphereEqn}, and \eqref{thirdSphereEqn} have no solutions when the points $(x_i,y_i,z_i),\ i=1,2,3$ are collinear. 

Subtracting \eqref{secondSphereEqn} from \eqref{firstSphereEqn}, and subtracting \eqref{thirdSphereEqn} from \eqref{firstSphereEqn}, we obtain the equations
\begin{align}
2(x,y,z)\cdot(x_2-x_1,\ y_2-y_1,\ z_2-z_1) & = x_1^2 + y_1^2 + z_1^2 - x_2^2 - y_2^2 - z_2^2,\label{firstPlaneEqn}\\
2(x,y,z)\cdot(x_3-x_1,\ y_3-y_1,\ z_3-z_1)& = x_1^2 + y_1^2 + z_1^2 - x_3^2 - y_3^2 - z_3^2\label{secondPlaneEqn}.
\end{align}

Note that the vectors $(x_2-x_1,\ y_2-y_1,\ z_2-z_1)$ and $(x_3-x_1,\ y_3-y_1,\ z_3-z_1)$ are parallel precisely if the points $(x_1,y_1,z_1)$, $(x_2,y_2,z_2)$, and $(x_3,y_3,z_3)$ are collinear (and thus there are no solutions to \eqref{firstSphereEqn},\eqref{secondSphereEqn}, and \eqref{thirdSphereEqn}).

If $(x_2-x_1,\ y_2-y_1,\ z_2-z_1)$ and $(x_3-x_1,\ y_3-y_1,\ z_3-z_1)$ are not parallel, then the set of points satisfying \eqref{firstSphereEqn}, \eqref{secondSphereEqn}, and \eqref{thirdSphereEqn} is precisely the set of points satisfying \eqref{firstSphereEqn}, \eqref{firstPlaneEqn}, and \eqref{secondPlaneEqn}; this is the intersection of a line and a sphere of radius $r^2$. Again, since a sphere of radius $r^2$ cannot contain any lines, this intersection consists of at most 2 points. To summarize, 
\begin{equation}\label{noK33}
|\{(x,y,z)\in \baseField^3\colon (x,y,z)\ \textrm{satisfies}\ \eqref{firstSphereEqn},\ \eqref{secondSphereEqn},\ \textrm{and}\ \eqref{thirdSphereEqn}\}|\leq 2.
\end{equation}

Let $\mathcal{L}_1=\{\phi(\mathbf p)\colon \mathbf p\in\mathcal{P}\}$, where $\phi$ is the map defined in Section \ref{lineSphereCorrespondenceSubsection}, and we identify a point $\mathbf p\in\mathcal{P}$ with the corresponding sphere of zero radius. Let $\mathcal{L}_2=\{\phi(S_{\mathbf p})\colon \mathbf p\in\mathcal{P}\}$, where $S_{\mathbf p}$ is the sphere with center $\mathbf p$ and radius $r$.

Apply Lemma \ref{Coloured2RichPointsInHeisenberg} to $\mathcal{L}_1\sqcup\mathcal{L}_2$, where the first set of lines is coloured red and the second set is coloured blue. The bound \eqref{noK33} shows that there cannot exist a doubly-ruled surface with three red lines in one ruling and three blue lines in the other ruling. We conclude that $\mathcal{L}_1\sqcup\mathcal{L}_2$ determines $O(n^{3/2})$ bichromatic 2-rich points. Thus there are $O(n^{3/2})$ pairs of points from $\mathcal{P}$ that satisfy \eqref{distanceR}.
\end{proof}

Finally, we will prove Theorem \ref{pointSphereThm}. The main observation is that while a pencil that is exactly $k$-rich determines $\binom{k}{2}$ pairs of tangent spheres, this pencil only determines $k-1$ or fewer point-sphere incidences (i.e.~at most $k-1$ tangencies between a sphere of zero radius and a sphere of nonzero radius).

\begin{proof}[Proof of Theorem \ref{pointSphereThm}]
Let $\mathcal{L}_1$ be the set of lines associated to the points from $\mathcal{P}$ (i.e.~spheres of radius 0), and let $\mathcal{L}_2$ be the set of lines associated to the spheres from $\mathcal{S}$.  Observe that by Remark \ref{allSameDifferentRadiiRem}, each $k$-rich pencil can contribute at most $k-1$ point-sphere incidences. 
Applying Lemma \ref{3RichPointsInHeisenberg} to $\mathcal{L}_1\sqcup\mathcal{L}_2$, we see that for each $k\geq 3$ the set of pencils of richness between $k$ and $2k$ can contribute at most $2kO(n^{3/2}k^{-3/2}) = O(n^{3/2}k^{-1/2})$ incidences. Summing dyadically in $k$, we conclude that the set of pencils of richness at least 3 can contribute $O(n^{3/2})$ incidences.

It remains to control the number of incidences arising from 2-rich points. Let $A = \lceil 2\eta^{-1}\sqrt n\ \rceil$. With this choice of $A$, apply Lemma \ref{Coloured2RichPointsInHeisenberg} to $\mathcal{L}_1\sqcup\mathcal{L}_2$, where the first set of lines is coloured red and the second set is coloured blue. If $\mathcal{L}_1\sqcup\mathcal{L}_2$ contains $O(An)$ bichromatic 2-rich points then we are done. If not, then there is a doubly-ruled surface $Z$ with one ruling that contains at least $A$ red lines and a second ruling that contains at least $A$ blue lines. Let $\mathcal{L}_2^\prime\subset\mathcal{L}_2$ be the set of blue lines from one of these rulings (recall that blue lines correspond to spheres from $\mathcal{S})$. Recall that each red line in $\mathcal{L}_1$ that is not contained in $Z$ intersects $Z$ in at most two points. Thus since $A\geq 2\eta^{-1}\sqrt n$, by pigeonholing there must exist a line $\ell\in\mathcal{L}_2^\prime$ that is incident to fewer than $\eta A$ lines from $\mathcal{L}_1$ that are not contained in $Z$. This implies that the corresponding sphere $S\in\mathcal{S}$ is \emph{not} $\eta$-non-degenerate, which contradicts the assumption that all of the spheres in $\mathcal{S}$ are $\eta$-non-degenerate.
\end{proof}

\section{Improvements over $\RR$}\label{incidenceGeomHeisenbergSec}
In this section we will show how Theorem \ref{spheresInGeneralFieldThm} can be improved when $\baseField=\RR$. The main tool will be the following polynomial partitioning theorem due to Guth \cite{Guth}. 
\begin{thm}\label{guthPartitioning}
Let $\mathcal{V}$ be a set of real algebraic varieties in $\RR^d$, each of which has dimension $e$ and is defined by a polynomial of degree at most $C$. Then for each $D\geq 1$, there is a $d$-variate ``partitioning'' polynomial $P$ of degree at most $D$ so that $\RR^d\backslash Z(P)$ is a disjoint union of $O(D^d)$ ``cells'' (open connected sets), and each of these cells intersect $O(|\mathcal{V}| D^{e-d})$ varieties from $\mathcal{V}$. The implicit constant depends on $d$ and $C$, but (crucially) is independent of $D$ and $|\mathcal{V}|$.
\end{thm}

While the definition of the dimension of a real algebraic variety is slightly technical, we will only use the elementary facts that points in $\RR^d$ are algebraic varieties of dimension 0 and lines are algebraic varieties of dimension 1. Applying Theorem \ref{guthPartitioning} to a set of points and to a set of lines in $\RR^3$ (with parameter $D/2$ in each case) and taking the product of the resulting partitioning polynomials, we obtain a polynomial whose zero-set efficiently partitions a set of points and a set of lines simultaneously. We will record this observation below. 
\begin{cor}\label{partitionPtsAndLines}
Let $\mathcal{P}$ be a set of points in $\RR^3$ and let $\mathcal{W}$ be a set of lines in $\RR^3$. Then for each $D\geq 1$, there is a polynomial $P$ of degree at most $D$ so that $\RR^3\backslash Z(P)$ is a disjoint union of $O(D^3)$ cells; each cell contains $O(|\mathcal{P}|D^{-3})$ points from $\mathcal{P}$; and each cell intersects $O(|\mathcal{V}| D^{-2})$ lines from $\mathcal{W}$.
\end{cor}

Note that while Corollary \ref{partitionPtsAndLines} guarantees that few points are contained in each cell and few lines meet each cell, it is possible that many points and lines are contained in the ``boundary'' $Z(P)$ of the cells. The next lemma helps us understand what configurations of points and lines are possible inside $Z(P)$.
\begin{lem}\label{boundaryBehavior}
Let $P\in \RR[x,y,z]$ be irreducible and let $D = \deg(P)$. If $Z(P)$ is not a plane, then there are at most $O(D^2)$ ``bad'' lines contained in $Z(P)$. If $\ell\subset Z(P)$ is not a bad line, then there are $O(D)$ ``bad'' points $\mathbf p\in\ell$. If $\mathbf q\in\ell$ is not a bad point, then it is incident to at most one additional line that is contained in $Z(P)$. 
\end{lem}
\begin{proof}
Statements of this form appear frequently in the literature (see e.g. Theorem 1.9 from \cite{SharirZlydenko}), and follow directly from the machinery of Guth and the author developed in \cite{GZ}. We will briefly outline the proof here. We say a point $\mathbf p\in Z(P)$ is 3-flecnodal if there are at least three distinct lines that contain $\mathbf p$ and are contained in $Z(P)$. If a line $\ell\subset Z(P)$ contains more than $CD$ 3-flecnodal lines for some absolute constant $C$, then all but finitely many points of $\ell$ must be 3-flecnodal (i.e.~a Zariski-dense subset of $\ell$ is 3-flecnodal). If there are $\geq CD^2$ lines $\ell\subset Z(P)$, each of which is 3-flecnodal at all but finitely many points, then there is a Zariski-dense subset $O\subset Z(P)$ so that for each point $\mathbf p\in O$ there are at least three lines containing $\mathbf p$ and contained in $Z(P)$. But since $P$ is irreducible, this immediately implies that $Z(P)$ is a plane. 
\end{proof}

The next two lemmas will help us understand the incidence geometry of configurations of coplanar lines coming from the Heisenberg group. Recall that throughout this section, $\baseField=\RR$ and $\extendedField=\CC$.
\begin{lem}\label{pointsOnALine}
Let $\mathcal{P}\subset\mathbb{H}$ be a set of points with $|\mathcal{P}|\geq 5$. Suppose that every pair of points is connected by a line in the Heisenberg group. Then $\mathcal{P}$ is contained in a (complex) line. 
\end{lem}
\begin{proof}
Let $\mathbf p, \mathbf q\in\mathcal{P}$.  Define $V_{\mathbf p} = \bigcup_{\mathbf p\in\ell\subset \mathbb{H}}\ell,$ and define $V_{\mathbf q}$ similarly. Then $V_{\mathbf p}$ (resp. $V_{\mathbf q}$) is precisely the intersection of $\mathbb{H}$ with the tangent plane of $\mathbb{H}$ at $\mathbf p$ (resp. $\mathbf q$). In particular, $V_{\mathbf p}\cap V_{\mathbf q}$ is contained in a complex line, and $\mathcal{P}\backslash \{\mathbf p,\mathbf q\}\subset V_{\mathbf p}\cap V_{\mathbf q}$. Since the choice of $\mathbf p$ and $\mathbf q$ was arbitrary, we conclude that every subset of $\mathcal{P}$ cardinality $|\mathcal{P}|-2$ is collinear. Since $|\mathcal{P}|\geq 5,$ this implies that all the points of $\mathcal{P}$ are collinear.
\end{proof}

\begin{lem}\label{preliminaryPtLineBd}
Let $\mathcal{P}$ be a set of $m$ points and let $\mathcal{L}$ be a set of $n$ lines in the Heisenberg group that are not parallel to the $xy$ plane. Then
\begin{equation}
I(\mathcal{P},\mathcal{L})=O(n^{3/2}+m).
\end{equation}
\end{lem}
\begin{proof}
By Corollary \ref{3RichPointsInHeisenberg}, for each $k\geq 3$ the number of $k$-rich points determined by $\mathcal{L}$ is $O(n^{3/2}k^{-3/2})$. Since a $k$-rich point contributes $k$ incidences, we conclude that the number of incidences coming from points with richness $\geq 3$ is $O(n^{3/2})$. The number of incidences coming from points with richness $\leq 2$ is at most $2m$.
\end{proof}

We are now ready to prove the main result of this section.

\begin{prop}\label{ptsLinesInHProp}
Let $\mathcal{P}$ be a set of $m$ points and let $\mathcal{L}$ be a set of $n$ lines in the Heisenberg group that are not parallel to the $xy$ plane. Then
\begin{equation}\label{inductionHypothesis}
I(\mathcal{P},\mathcal{L})\leq C (m^{3/5}n^{3/5}+m+n),
\end{equation}
where $C$ is an absolute constant
\end{prop}
\begin{proof}
Our proof is closely modeled on the techniques of Sharir and Zlydenko from \cite{SharirZlydenko}. We will prove the result by induction on $m$; the base case for our induction will be when $m\leq m_0$, where $m_0$ is an absolute constant to be specified below. Observe that if $m^2 \geq cn^3$ for a fixed constant $c>0$, then then the result follows from Lemma \ref{preliminaryPtLineBd}; indeed, if $m^2 \geq cn^3$ then $n^{3/2}+m=O(m)$ and thus
$$
I(\mathcal{P},\mathcal{L})=O(n^{3/2}+m) = O(m^{3/5}n^{3/5}+m+n).
$$
Henceforth we will assume that 
\begin{equation}\label{msqleqnCub}
m^2 \leq cn^3,
\end{equation}
where $c$ is a constant that will be specified below.

For each line $\ell\in\mathcal{L}$ of the form $(0,\omega,t)+\CC(1,u,\bar\omega)$, let $\mathbf{q}_{\ell}\in\RR^4$ be the point $(t,u,\operatorname{Re} \omega,\operatorname{Im} \omega)$. Define
$$
\mathcal{Q} = \{\mathbf{q}_{\ell}\colon \ell\in\mathcal{L}\}.
$$  
For each point $\mathbf{p}\in\mathcal{P}$, let $W_{\mathbf{p}}\subset\RR^4$ be the line described by \eqref{linesThruAPt}. Define 
$$
\mathcal{W} = \{W_{\mathbf{p}}\colon \mathbf p\in\mathcal{P}\}.
$$ 
By Lemma \ref{preliminaryPtLineBd}, for any set of points $\mathcal{Q}^\prime\subset\mathcal{Q}$ and any set of lines $\mathcal{W}^\prime\subset\mathcal{W}$, we have the incidence bound
\begin{equation}\label{prelimPtLineIncidenceBoundDual}
I(\mathcal{Q}^\prime,\mathcal{W}^\prime) = O(|\mathcal{Q}^\prime|^{3/2}+|\mathcal{W}^\prime|).
\end{equation}

Let $A\colon\RR^4\to\RR^3$ be a surjective linear map; we will choose this map so that the image of every line in $\mathcal{W}$ remains a line, and no additional incidences are added. We will abuse notation slightly by replacing $\mathcal{Q}$ with the set $\{A(q)\colon q\in\mathcal{Q}\}$ (so now $\mathcal{Q}\subset\RR^3$) and replacing $\mathcal{W}$ with the set $\{A(W)\colon W\in\mathcal{W}\}$ (so now $\mathcal{W}$ is a set of lines in $\RR^3$). Note that \eqref{prelimPtLineIncidenceBoundDual} remains true with these new definitions of $\mathcal{Q}$ and $\mathcal{W}$. 

Define 
\begin{equation}\label{defnD}
D = \lfloor c\min(n^{3/5}m^{-2/5}, m^{1/2})\rfloor,
\end{equation}
where $c>0$ is the same constant as in \eqref{msqleqnCub}. In the analysis that follows, ``Case 1'' will refer to the situation where $n^{3/5}m^{-2/5}\geq m^{1/2}$, and ``Case 2'' will refer to the situation where $n^{3/5}m^{-2/5}< m^{1/2}$.

Observe that $D\leq cn^{1/3}.$ Indeed, in Case 1 we have $m\geq n^{2/3}$ and thus $D\leq n^{3/5}m^{-2/5}\leq cn^{1/3}$. In Case 2 we have $m\leq n^{2/3}$, and again $D\leq cm^{1/2}\leq cn^{1/3}$. If we select $m_0\geq c^{-2}$, then we have ensured that $cm^{1/2}> 1$. Inequality \eqref{msqleqnCub} implies that $cn^{3/5}m^{-2/5}\geq 1$; thus we can assume that $D\geq 1$. In summary, we have
\begin{equation}\label{sizeOfD}
1\leq D\leq cn^{1/3}.
\end{equation}
Apply Corollary \ref{partitionPtsAndLines} to $\mathcal{Q}$ and $\mathcal{W}$ with this choice of $D$; we obtain a partitioning polynomial $P\in\RR[x,y,z]$. $\RR^3\backslash Z(P)$ is a union of $O(D^3)$ cells, each of which contains $O(n/D^3)$ points from $\mathcal{Q}$ and each of which intersects $O(m/D^2)$ line from $\mathcal{W}$.  

In Case 1, we use \eqref{prelimPtLineIncidenceBoundDual} to control the number of point-line incidences inside the cells. If we index the cells $O_1,\ldots,O_t$ with $t = O(D^3)$ and define $\mathcal{W}_{i}$ to be the set of lines from $\mathcal{W}$ that intersect $O_i$, then
\begin{equation}\label{inCellBoundOne}
\begin{split}
I(\mathcal{Q}\backslash Z(P),\mathcal{W}) & = \sum_i I\big( \mathcal{Q}\cap O_i,\ \mathcal{W}_{i}\big)\\
& \lesssim D^3 \big((nD^{-3})^{3/2}+ mD^{-2}\big) \\
& \lesssim n^{3/2}D^{-3/2}+mD\\
& \lesssim  m^{3/5}n^{3/5}.
\end{split}
\end{equation}
In Case 2 each cell meets $O(1)$ lines, and thus
\begin{equation}\label{inCellBoundTwo}
I(\mathcal{Q}\backslash Z(P),\mathcal{W})  = \sum_i  I\big( \mathcal{Q}\cap O_i,\ \mathcal{W}_{i}\big) \lesssim |\mathcal{Q} \backslash Z(P)|.
\end{equation}

Let $\mathcal{W} = \mathcal{W}^\prime \sqcup\mathcal{W}^{\prime\prime}$, where $\mathcal{W}^\prime$ consists of the lines not contained in $Z(P)$ and $\mathcal{W}^{\prime\prime}$ consists of the lines contained in $Z(P)$. Since each line in $\mathcal{W}^\prime$ can intersect $Z(P)$ at most $D$ times, we have
\begin{equation}\label{incidencesBoundaryNonContainedLines}
I(\mathcal{Q}\cap Z(P),\ \mathcal{W}^\prime)\leq Dm\leq cm^{3/5}n^{3/5}.
\end{equation}

Our next task is to control the number of incidences formed by lines in $\mathcal{W}^{\prime\prime}$. Factor $P$ into irreducible components $P_1\cdots P_h\cdot P_{h+1}\cdots P_k$, where the polynomials $P_1,\ldots,P_h$ each have degree at least two and the polynomials $p_{h+1}\ldots p_k$ have degree one (it is possible that all polynomials have degree at least two or no polynomials have degree at least two. In the former case we set $k=h$ and in the latter case we set $h = -1$). For each index $i$, define
$$
\mathcal{Q}_i=\mathcal{Q}\cap Z(P_i)\ \backslash\bigcup_{1\leq j<i}\mathcal{Q}_j.
$$
We have $\mathcal{Q}\cap Z(P) = \bigsqcup_{i=1}^k \mathcal{Q}_i$. Similarly, let $\mathcal{W}_i$ be the set of lines that are contained in $Z(P_i)$ and are not contained in $Z(P_j)$ for any index $j<i$. 

If $\mathbf{q}\in\ell$ for some $\mathbf q\in\mathcal{Q}_i$ and $\ell\in \mathcal{W}_j$ with $i\neq j$, then the line $\ell$ must intersect $Z(P_i)$, and $\ell$ cannot be contained in $Z(P_i)$. Since $\ell$ can intersect $Z(P_i)$ in at most $\deg(P_i)$ points, we can bound the number of such ``cross-index'' incidences as follows.
\begin{equation}\label{incidencesCrossIndexNonPlanar}
\begin{split}
|\{(\mathbf q,\ell)\colon & \mathbf q\in\ell,\ \mathbf q\in\mathcal{Q}_i,\ \ell\in \mathcal{W}_j\ \textrm{for some index}\ j\neq i\}|\\
&\leq \sum_{i=1}^k m\operatorname{deg}(P_i)\\
&\leq Dm\\
&\leq cm^{3/5}n^{3/5}.
\end{split}
\end{equation}
It remains to control incidences $\mathbf p\in \ell$ where $\mathbf p\in\mathcal{Q}_i$ and $\ell\in\mathcal{W}_i$. Let $\mathcal{Q}_{i}^{\operatorname{rich}}$ consist of those points $\mathbf q \in\mathcal{Q}_i$ that are incident to at least 3 lines from $\mathcal{W}_i$. Define $\mathcal{Q}_{i}^{\operatorname{poor}} = \mathcal{Q}_i\backslash \mathcal{Q}_{i}^{\operatorname{rich}}$. By Lemma \ref{boundaryBehavior}, for each index $1\leq i\leq h$ there are $\leq C_1 \deg(P_i)^2$ ``bad'' lines for some absolute constant $C_1$. For each index $1\leq i\leq h$, let $\mathcal{W}_i^{\operatorname{bad}}\subset\mathcal{W}_i$ be the set of bad lines associated to $P_i$ and let $\mathcal{W}_i^{\operatorname{good}}=\mathcal{W}_i\backslash\mathcal{W}_i^{\operatorname{bad}}$. If we choose the constant $c$ from \eqref{defnD} sufficiently small, then 
$$
\sum_{i=1}^h |\mathcal{W}_i^{\operatorname{bad}}|\leq \sum_{i=1}^h C_1 \deg(P_i)^2 \leq C_1D^2 \leq  m/2.
$$ 
Thus by the induction hypothesis we have
\begin{equation*}
\begin{split}
\sum_{i=1}^h I(\mathcal{Q}_{i}^{\operatorname{rich}},\mathcal{W}_i^{\operatorname{bad}})&\leq
I\Big(\bigcup_{i=1}^h\mathcal{Q}_{i}^{\operatorname{rich}},\ \bigcup_{i=1}^h\mathcal{W}_i^{\operatorname{bad}}\Big)\\
& \leq C \Big(\big(\frac{m}{2}\big)^{3/5}n^{3/5}+\big|\bigcup_{i=1}^h\mathcal{W}_i^{\operatorname{bad}}\big|+\big|\bigcup_{i=1}^h\mathcal{Q}_{i}^{\operatorname{rich}}\big|\Big).
\end{split}
\end{equation*}
If $\ell\in \mathcal{W}_i^{\operatorname{good}}$ then $\ell$ can be incident to at most $C_2D$ bad points, for some absolute constant $C_2$. Since 
$$
\sum_{i=1}^h C_2D|\mathcal{W}_i^{\operatorname{good}}| = O(m^{3/5}n^{3/5}),
$$ 
we have 
\begin{equation*}
\sum_{i=1}^h I(\mathcal{Q}_{i}^{\operatorname{rich}},\mathcal{W}_i) \leq C \Big(\big(\frac{m}{2}\big)^{3/5}n^{3/5}+\big|\bigcup_{i=1}^h\mathcal{W}_i^{\operatorname{bad}}\big|+\big|\bigcup_{i=1}^h\mathcal{Q}_{i}^{\operatorname{rich}}\big|\Big)+O(m^{3/5}n^{3/5}).
\end{equation*}
If $\mathbf{q}\in\mathcal{Q}_{i}^{\operatorname{poor}}$ then $\mathbf{q}$ can be incident to at most two lines from $\mathcal{W}_i$. Thus as long as $C\geq 2$ we have  

\begin{equation}\label{incidencesMatchingIndexNonPlanar}
\sum_{i=1}^h I(\mathcal{Q}_{i},\mathcal{W}_i) \leq C \Big(\big(\frac{m}{2}\big)^{3/5}n^{3/5}+\sum_{i=1}^h |\mathcal{W}_i|+\sum_{i=1}^h |\mathcal{Q}_i| \Big)+O(m^{3/5}n^{3/5}).
\end{equation}

It remains to bound $I(\mathcal{Q}_i,\mathcal{W}_i)$ for $h+1\leq i\leq k$. By Lemma \ref{pointsOnALine}, if $|\mathcal{W}_i|\geq 5$, then the lines in $\mathcal{W}_i$ correspond to a set of points in $\mathbb{H}$ that are collinear. In particular, this implies
$$
I(\mathcal{Q}_i,\mathcal{W}_i)\leq |\mathcal{Q}_i| + |\mathcal{W}_i|.
$$
On the other hand, if $|\mathcal{W}_i|\leq 4$ then
$$
I(\mathcal{Q}_i,\mathcal{W}_i) \leq 4|\mathcal{Q}_i|.
$$
Combining these bounds with \eqref{incidencesMatchingIndexNonPlanar} and summing in $i$, we conclude that if $C\geq 4$ then 
\begin{equation}\label{incidencesMatchingIndexPlanarAndNonPlanar}
\begin{split}
\sum_{i=1}^k I(\mathcal{Q}_{i},\mathcal{W}_i) &\leq C \Big(\big(\frac{m}{2}\big)^{3/5}n^{3/5}+\sum_{i=1}^h |\mathcal{W}_i|+\sum_{i=1}^h |\mathcal{Q}_i| \Big) + C\Big( \sum_{i=h+1}^k|\mathcal{W}_i|+\sum_{i=h+1}^k|\mathcal{Q}_i|\Big) \\
&\leq C \Big(\big(\frac{m}{2}\big)^{3/5}n^{3/5}+m+n \Big).
\end{split}
\end{equation}

Combining \eqref{inCellBoundOne}, \eqref{inCellBoundTwo}, \eqref{incidencesBoundaryNonContainedLines}, \eqref{incidencesCrossIndexNonPlanar}, and \eqref{incidencesMatchingIndexPlanarAndNonPlanar}, we conclude that there is an absolute constant $C_2$ so that
\begin{equation}\label{combiningBound}
I(\mathcal{Q},\mathcal{W})\leq C \Big(\big(\frac{m}{2}\big)^{3/5}n^{3/5}+m+n \Big) + C_2m^{3/5}n^{3/5}.
\end{equation}
If the constant $C$ is chosen sufficiently large so that $C  \geq 2^{3/5}C_2$, then we obtain \eqref{inductionHypothesis} and the induction closes.
\end{proof}

\begin{proof}[Proof of Theorem \ref{spheresInR3Thm}]
Let $\mathcal{L}=\{\phi(S)\colon S\in\mathcal{S}\}$ and let $k\geq 3$. We need to prove that $\mathcal{L}$ determines $O(n^{3/2}k^{-5/2}+n/k)$ $k$-rich points. When $k$ is small the result follows from Theorem \ref{spheresInGeneralFieldThm}, so we can assume that $k\geq C$, for some absolute constant $C$ to be determined later. Let $\mathcal{P}\subset \mathbb{H}$ be the set of $k$-rich points determined by $\mathcal{L}$, and let $m = |\mathcal{P}|$. We clearly have $I(\mathcal{P},\mathcal{L})\geq km$. On the other hand, Proposition \ref{ptsLinesInHProp} implies that
\begin{equation}\label{numberOfPtLineIncidencesApp}
I(\mathcal{P},\mathcal{L})=O(m^{3/5}n^{3/5}+m+n). 
\end{equation}
Thus if the constant $C$ is selected sufficiently large compared to the implicit constant in \eqref{numberOfPtLineIncidencesApp}, then $m=O(n^{3/2}k^{-5/2}+nk^{-1})$, as desired.
\end{proof}

\bibliographystyle{abbrv}
\bibliography{sphere_tangencies}

\end{document}